\documentclass[11pt,a4paper]{article}

\usepackage{inputenc}
\usepackage{amsmath}
\usepackage{bm}
\usepackage{bbold}
\usepackage{amsthm}

\usepackage{hyperref}

\title{A Max-Algebra Approach to Modeling\\ and Simulation of Tandem Queueing Systems\thanks{Mathematical and Computer Modelling, 1995. Vol.~22, no.~3, pp.~25-31.}}

\author{N. K. Krivulin\thanks{Faculty of Mathematics and Mechanics, St.~Petersburg State University, 28 Universitetsky Ave., St.~Petersburg, 198504, Russia, 
nkk@math.spbu.ru}
\thanks{
The writing of this paper was completed while the author was visiting the
Center for Economic Research (CentER) at Tilburg University, Tilburg, The
Netherlands. The author thanks the CentER for its hospitality. He is also very
much grateful to J.~P.~C.~Kleijnen for valuable comments and suggestions which
allow the author to improve the clarity of presentation.}
}

\date{}

\newtheorem{theorem}{Theorem}
\newtheorem{lemma}[theorem]{Lemma}

\def\sumo_#1^#2{\setbox0=\hbox{$\displaystyle{\sum}$}
                \setbox1=\hbox{$\scriptstyle{#1}$}
                \setbox2=\hbox{$\scriptstyle{#2}$}
		\setbox3=\hbox{${}_{{}_\oplus}\mathsurround=0pt$}
		\dimen1=.5\wd1 \advance\dimen1 by-.5\wd0
		\ifdim\dimen1>0pt
		   \ifdim\dimen1>\wd3 \kern\wd3 \else\kern\dimen1\fi\fi
		\dimen2=.5\wd2 \advance\dimen2 by-.5\wd0
		\ifdim\dimen2>0pt
		   \ifdim\dimen2>\wd3 \kern\wd3 \else\kern\dimen2\fi\fi
		\mathop{{\sum}{}_{{}_\oplus}}_{\kern-\wd3 #1}^{\kern-\wd3 #2}}

\newenvironment{algorithm}[1]{\vspace{1em}{\noindent\bf Algorithm #1.}\bf
                                \begin{tabbing}}{\end{tabbing}}

\begin{document}

\maketitle

\begin{abstract}
Max-algebra models of tandem single-server queueing systems with both finite
and infinite buffers are developed. The dynamics of each system is described
by a linear vector state equation similar to those in the conventional linear
systems theory, and it is determined by a transition matrix inherent in the
system. The departure epochs of a customer from the queues are considered as
state variables, whereas its service times are assumed to be system
parameters. We show how transition matrices may be calculated from the service
times, and present the matrices associated with particular models. We also
give a representation of system performance measures including the system time
and the waiting time of customers, associated with the models. As an
application, both serial and parallel simulation procedures are presented,
and their performance is outlined.
\\

\end{abstract}

\section{Introduction}

In the analysis of queueing systems, algebra models arise naturally from the
recursive equations of the Lindley type, which present a formalism widely used
for the representation of dynamics of a variety of queueing systems models.
There are the recursive equations designed to describe the $ G/G/m $
queue, closed and open tandem queueing systems which may have both infinite
and finite buffers, and queueing networks with deterministic routing (see
e.g., \cite{1,2,3,4,5,6}). These equations, which allow the dynamics of a queueing
system to be represented in a convenient and unified way well suited to
analytical treatments, also provide the basis for the development of efficient
procedures of queueing system simulation \cite{1,2,3,7}.

Since recursive equations often involve only the operations of arithmetic
addition and maximization, they offer the prospect of the representation of
queueing system models in terms of the {\it max-algebra theory\/}
\cite{8,9,10}. The implementation of max-algebra allows one to rewrite the
recursive equations as linear scalar and vector algebraic equations \cite{11},
which are actually almost identical to those in the conventional linear system
theory. The benefits of the max-algebra approach in the analysis of queueing
systems are twofold: first, it gives us the chance to exploit results of the
conventional linear algebra, which have been reformulated, and are now
available in the max-algebra. The classical results already reformulated and
proved in the max-algebra include, in particular, the solution of the
eigenvalue problem, the Cayley-Hamilton theorem, and Cramer's rule
\cite{8,9,10,12}.

Other benefits have a direct relationship to computational aspects of
simulation. In fact, the algebraic models of queues lead to matrix-vector
max-algebra multiplications as the basis of simulation procedures \cite{2}.
New possibilities then arise in queueing system simulation to employ efficient
computational methods and algorithms available in numerical algebra, including
those designed for implementation on parallel and vector processors.

In this paper we develop max-algebra models of open and closed tandem
single-server queueing systems which may have both infinite and finite
buffers, and we give related representations of system performance measures.
We start with preliminary algebraic definitions in Section~2 which also
includes a technical lemma underlying the development of models in later
sections. In Section~3, the dynamics of open and closed systems with infinite
buffers is described by a vector state equation which is determined by a
transition matrix inherent in the system. The departure epochs of a customer
from the queues are considered as state variables, whereas its service times
are assumed to be system parameters. We show how transition matrices may be
calculated from the service times, and present the matrices associated with
certain particular models.

Section~4 extends the dynamic equation to cover open tandem systems with
finite buffers, which operate under both manufacturing and communication
blocking rules. The representations of system performance measures including
the system times and the waiting times of customers, associated with the
models are given in Section~5. In Section~6, we present serial and parallel
simulation algorithms based on the algebraic models, and outline their
performance. Finally, Section~7 gives conclusions.

\section{Preliminary Algebraic Definitions and Results}

In this section we briefly outline basic facts about matrix max-algebra, which
underlie the algebraic models and methods of queueing system simulation,
presented in the subsequent sections. Further details concerning the
max-algebra and its applications can be found in survey papers \cite{10,12}. A
thorough theoretical analysis of this algebra and related algebraic systems is
given in \cite{8,9}.

We start with the max-algebra of real numbers, which is the system
$ (\underline{\mathbb{R}}, \oplus, \otimes) $, where
$ \underline{\mathbb{R}} = \mathbb{R} \cup \{\varepsilon\} $ with
$ \varepsilon = -\infty $, and
$$
x \oplus y = \max(x,y), \qquad x \otimes y = x + y
$$
for any $ x,y \in \mathbb{R} $.

It is easy to see that these new operations, namely addition $ \oplus $
and product $ \otimes $, possess the following properties:
\begin{alignat*}{2}
& \text{\it Associativity:} & \quad &
                         x \oplus (y \oplus z) = (x \oplus y) \oplus z, \\
& & &                       
                         x \otimes (y \otimes z) = (x \otimes y) \otimes z; \\
& \text{\it Commutativity:} &\quad & x \oplus y = y \oplus x, \quad
                            x \otimes y = y \otimes x; \\
& \text{\it Distributivity:} & \quad &
	     x \otimes (y \oplus z) = (x \otimes y) \oplus (x \otimes z); \\
& \text{\it Idempotency of Addition:} & \quad & x \oplus x = x.
\intertext{With $ e = 0 $, we further have}
& \text{\it Null and Identity Elements:} & \quad &
                       x \oplus \varepsilon = \varepsilon \oplus x = x, \quad
                       x \otimes e = e \otimes x = x; \\
& \text{\it Absorption Rule:} & \quad &
x \otimes \varepsilon = \varepsilon \otimes x = \varepsilon.
\end{alignat*}
Clearly, in the max-algebra these properties allow ordinary algebraic
manipulation of expressions involving the max-algebra operations to be
performed under the usual conventions regarding brackets and precedence of
$ \otimes $ over $ \oplus $.

Note finally that in the max-algebra, for each $ x \in \mathbb{R} $, there
exists its multiplicative inverse $ x^{-1} $ such that
$ x \otimes x^{-1} = x^{-1} \otimes x = e $. This is the usual arithmetic
inverse which satisfies, in particular, the evident condition
$$
(x \otimes y)^{-1} = x^{-1} \otimes y^{-1}
$$
for all $ x,y \in \mathbb{R} $.

\subsection{Max-algebra of Matrices}

The scalar max-algebra is extended to the max-algebra of matrices in the
regular way. Specifically, for any square $(n\times n)$-matrices
$ A = (a_{ij}) $ and $ B = (b_{ij}) $ with entries in
$ \underline{\mathbb{R}} $, the elements of the matrices $ C = A \oplus B $
and $ D = A \otimes B $ are calculated as
$$
c_{ij} = a_{ij} \oplus b_{ij}, \qquad \text{and} \qquad
d_{ij} = \sumo_{k=1}^{n} a_{ik} \otimes b_{kj},
$$
where $ \sum_{\oplus} $ denotes the iterated operation $ \oplus $,
$ i=1,\ldots,n $; $ j=1,\ldots,n $. Similarly, the multiplication of a
matrix by a scalar, as well as the operations of both matrix-vector
multiplication and vector addition may be routinely defined.

As in the scalar max-algebra, there are null and unit elements in the matrix
algebra, defined respectively as
$$
\mathcal{E} = \left(
             \begin{array}{ccccc}
	       \varepsilon & \ldots & \varepsilon \\
	       \vdots      & \ddots & \vdots \\
	       \varepsilon & \ldots & \varepsilon
	     \end{array}
	   \right), \qquad
       E = \left(
             \begin{array}{ccccc}
	       e           &        & \varepsilon \\
	                   & \ddots & \\
	       \varepsilon &        & e
	     \end{array}
	   \right).
$$
One can easily see that for any square matrix $ A $, it holds
$$
\mathcal{E} \otimes A = A \otimes \mathcal{E} = \mathcal{E}, \qquad
E \otimes A = A \otimes E = A, \qquad
\mathcal{E} \oplus A = A \oplus \mathcal{E} = A.
$$

It is not difficult to verify that the other properties of scalar operations
$ \oplus $ and $ \otimes $, with the exception of the commutativity of
multiplication, are also extended to the matrix algebra. Similar to the
conventional matrix algebra, matrix multiplication in the max-algebra is not
commutative in general. Furthermore, the multiplicative inverse does not
generally exists in this matrix algebra. However, one can easily obtain the
inverse of any diagonal square matrix $ A $ with entries not equal to
$ \varepsilon $ only on the diagonal. It is clear that with
$$
A = \left(
      \begin{array}{ccccc}
        a_{1}       &        & \varepsilon \\
	            & \ddots & \\
	\varepsilon &        & a_{n}
      \end{array}
     \right), \qquad
A^{-1} = \left(
          \begin{array}{ccccc}
            a_{1}^{-1}  &        & \varepsilon \\
	                & \ddots & \\
	    \varepsilon &        & a_{n}^{-1}
	  \end{array}
         \right),
$$
where $ a_{i} > \varepsilon $ for all $ i=1, \ldots, n $, we have
$ A \otimes A^{-1} = A^{-1} \otimes A = E $.

Finally, the properties of matrix multiplication allow us to exploit the
symbol $ A^{p} $ with a square matrix $ A $ and nonnegative integer
$ p $, as used in the conventional algebra:
$$
A^{e} = E, \qquad
A^{p} = A^{p-1} \otimes A
= \underbrace{A \otimes \cdots \otimes A}_{\text{$p $ times}}
\quad \text{for $ p \geq 1$}.
$$

\subsection{A Linear Algebraic Equation}

Let us now examine a vector equation which will be encountered below in
algebraic representations of tandem queueing system dynamics. For given
$(n\times n)$-matrix $ A $ and column $n$-vector $ \bm{b} $, we
consider the implicit equation in the $n$-vector $ \bm{x} $
\begin{equation}
\bm{x} = A \otimes \bm{x} \oplus \bm{b}, \label{tag1}
\end{equation}
which is generally identified analogous to the conventional linear algebra as
a linear equation.

The next lemma offers the solution of \eqref{tag1} for a particular class of
matrices $ A $. A detailed investigation of this and other linear equations
in the general case can be found in \cite{9}.
\begin{lemma}
If there exists a nonnegative integer $ p $ such that
$ A^{p} = \mathcal{E} $, then \eqref{tag1} has the unique solution
\begin{equation}
\bm{x} = \sumo_{i=0}^{p-1} A^{i} \otimes \bm{b}. \label{tag2}
\end{equation}
\end{lemma}

\begin{proof}
Recurrent substitutions of $ \bm{x} $ from \eqref{tag1}
into the right-hand side of \eqref{tag1} and trivial algebraic manipulations
give
\begin{multline*}
\bm{x} = A \otimes \bm{x} \oplus \bm{b}
= A \otimes (A \otimes \bm{x} \oplus \bm{b}) \oplus \bm{b}
     = A^{2} \otimes \bm{x} \oplus ( E \oplus A) \otimes \bm{b} \\
= \cdots = A^{p} \otimes \bm{x}
        \oplus (E \oplus A \oplus \cdots \oplus A^{p-1}) \otimes \bm{b}.
\end{multline*}
With the condition $ A^{p} = \mathcal{E} $, we immediately arrive at
\eqref{tag2}. It is also evident from the above calculations that the obtained
solution is unique.
\end{proof}

As examples of the matrix $ A $ satisfying the condition of Lemma~1, we
consider either lower or upper triangular matrices which have, in addition,
the entries equal to $ \varepsilon $ on the main diagonal. Specifically,
it is not difficult to verify that for the matrix $ A $ with entries
$$
a_{ij} = \begin{cases}
	   \alpha_{i}> \varepsilon, & \text{if $ i=j+1$} \\
	   \varepsilon,             & \text{otherwise},
	 \end{cases}
$$
which will appear in the next sections, it holds that
$ A^{p} \neq \mathcal{E} $, for $ p=1,\ldots,n-1 $, and
$ A^{n} = \mathcal{E} $.

\section{Representation of Tandem System Dynamics}

We start with the scalar max-algebra equation representing the dynamics of a
single-server queue, and then extend it to vector equations associated with
tandem systems of queues. The queue is assumed to have a buffer with infinite
capacity, and to operate under the first-come, first-served queue discipline.

We denote the $k^{\text{th}}$ arrival epoch to the queue and the
$k^{\text{th}}$ departure epoch from the queue by $ a(k) $ and
$ d(k) $ respectively. The service time of customer $ k $ is
represented by $ \tau_{k} $. Under the conditions that the queue starts
operating at time zero and it has no customers at the initial time, the
dynamics of the system can be readily described in terms of $ a(k) $ and
$ d(k) $ as state variables, by the following ordinary algebraic equation
\cite{1,2,6}:
$$
d(k) = \max(a(k),d(k-1)) + \tau_{k}.
$$
By replacing the usual operation symbols by those of max-algebra, one can
rewrite this equation in its equivalent form as
$$
d(k)=\tau_{k} \otimes a(k) \oplus \tau_{k} \otimes d(k-1).
$$

Let us now suppose that there is a system of $ n $ single-server queues
with infinite buffers. Furthermore, let $ a_{i}(k) $ and $ d_{i}(k) $
be the $k^{\text{th}}$ arrival and departure epochs, and $ \tau_{ik} $ be
the service time of the $k^{\text{th}}$ customer for queue $ i $. With the
vector-matrix notations
$$
\bm{a}(k) = \left(
                    \begin{array}{ccccc}
	              a_{1}(k) \\
		      \vdots   \\
	              a_{n}(k)
		    \end{array}
	          \right), \qquad
\bm{d}(k) = \left(
                    \begin{array}{ccccc}
	              d_{1}(k) \\
		      \vdots   \\
	              d_{n}(k)
		    \end{array}
	          \right), \qquad
\mathcal{T}_{k} = \left(
                 \begin{array}{ccccc}
	           \tau_{1k}   &        & \varepsilon \\
		               & \ddots &             \\
	           \varepsilon &        & \tau_{nk}
		 \end{array}
	       \right),
$$
and the conditions $ \bm{d}(0) = (e, \ldots, e)^{T} $, and
$ \bm{d}(k) = (\varepsilon, \ldots, \varepsilon)^{T} $ for all
$ k < 0 $, we may represent the dynamics of the whole system by the
vector equation
\begin{equation}
\bm{d}(k) = \mathcal{T}_{k} \otimes \bm{a}(k)
                          \oplus \mathcal{T}_{k} \otimes \bm{d}(k-1). \label{tag3}
\end{equation}

Note that equation \eqref{tag3} is quite general; it may be appropriate for a
variety of single-server queueing systems with infinite buffers, and not just
for tandem systems. We will show how this equation can be refined in the case
of two particular tandem systems in the next subsections.

\subsection{Closed Tandem Systems}

Consider a closed system of $ n $ single-server queues in tandem, with
infinite buffers. In the system, the customers have to pass through the queues
consecutively so as to receive service at each server. After their service
completion at the $n^{\text{th}}$ server, the customers return to the first
queue for a new cycle of service. It is assumed that the transition of
customers from a queue to the next one requires no time.

There are a finite number of customers circulating through the system. We
suppose that at the initial time, all servers are free, whereas the buffer at
the $i^{\text{th}}$ server contains $ c_{i} $ customers,
$ i=1, \ldots, n$. It is easy to see that in this case, the arrival times of
customers can be defined as
\begin{equation}
a_{i}(k) = \begin{cases}
              d_{n}(k-c_{1}),   & \text{if $ i=1$} \\
              d_{i-1}(k-c_{i}), & \text{if $ i=2,\ldots,n$},
	   \end{cases} \label{tag4}
\end{equation}
for all $ k=1,2,\ldots $.

Let us now assume for simplicity that $ c_{i} = 1 $ for all
$ i=1,\ldots,n $. With this assumption, one can rewrite \eqref{tag4} in the
vector form
$$
\bm{a}(k) = F \otimes \bm{d}(k-1), \qquad
\text{where} \qquad
F = \left(
      \begin{array}{ccccc}
        \varepsilon & \cdots & \varepsilon & e  \\
	e	    & \ddots & \varepsilon & \varepsilon \\
		    & \ddots & \ddots      & \vdots \\
	\varepsilon &        & e           & \varepsilon
      \end{array}
    \right).
$$
Substitution of this expression for $ \bm{a}(k) $ into \eqref{tag3}
leads to
$$
\bm{d}(k)
= \mathcal{T}_{k} \otimes F \otimes \bm{d}(k-1)
                              \oplus \mathcal{T}_{k} \otimes \bm{d}(k-1)
= \mathcal{T}_{k} \otimes (F \oplus E) \otimes \bm{d}(k-1).
$$
We may now describe the dynamics of the closed system under consideration by
the linear state equation
\begin{equation}
\bm{d}(k) = T_{k} \otimes \bm{d}(k-1), \label{tag5}
\end{equation}
with the state transition matrix
\begin{equation}
T_{k} = \mathcal{T}_{k} \otimes (F \oplus E)
= \left(
    \begin{array}{ccccc}
      \tau_{1k}   & \varepsilon & \cdots & \varepsilon & \tau_{1k}   \\
      \tau_{2k}   & \tau_{2k}   &        & \varepsilon & \varepsilon \\
                  & \ddots      & \ddots &             & \vdots      \\
      \varepsilon & \varepsilon & \ddots & \tau_{n-1k} & \varepsilon \\
      \varepsilon & \varepsilon &        & \tau_{nk}   & \tau_{nk}
    \end{array}
  \right). \label{tag6}
\end{equation}
One can see that equation \eqref{tag5} offers a very convenient way of
calculating successive state vectors $ \bm{d}(k) $ in the above
system, by performing simple algebraic operations. Moreover, it is not
difficult to understand that \eqref{tag5} can be readily extended to represent
closed tandem systems with arbitrary $ c_{i} $, $ 0 \leq c_{i} < \infty $.
A usual technique based on the employment of a modified state vector
$ \widetilde{\bm{d}}(k) $ which integrates several consecutive state
vectors of the original model, can be exploited to get \eqref{tag5} as a
representation of such systems.

To illustrate, let us suppose that at the initial time, there are
$ c_{i} = 2 $ customers in the buffer at each queue $ i $,
$ i=1,\ldots,n $. In this case, starting from \eqref{tag4}, we may define the
vector of the $k^{\text{th}}$ arrival epochs as
$$
\bm{a}(k) = F \otimes \bm{d}(k-2),
$$
and then rewrite \eqref{tag3} in the form
$$
\bm{d}(k) = \mathcal{T}_{k} \otimes \bm{d}(k-1)
                      \oplus \mathcal{T}_{k} \otimes F \otimes \bm{d}(k-2).
$$

Finally, with the new state vector
$ \widetilde{\bm{d}}(k) = \left(\begin{array}{l}
                                      \bm{d}(k) \\
                                      \bm{d}(k-1)
				    \end{array}\right) $, it is not difficult to
arrive at the state equation
$$
\widetilde{\bm{d}}(k) = \widetilde{T}_{k}
                                        \otimes \widetilde{\bm{d}}(k-1),
$$
with the state transition matrix
$$
\widetilde{T}_{k} = \left(
                      \begin{array}{ccccc}
                        \mathcal{T}_{k} & \mathcal{T}_{k} \otimes F  \\
	                E            & \mathcal{E}
		      \end{array}
                    \right).
$$

\subsection{Open Tandem Systems}

The purpose of this subsection is to extend equation \eqref{tag5} to the
representation of {\it open\/} tandem queueing system models. In a series of
single-server queues with infinite buffers, let us assign the first queue for
representing an external arrival stream of customers. Each customer that
arrives into the system has to pass through queues $ 2 $ to $ n $, and
then leaves the system.

As in the closed system above, we suppose that $ \tau_{ik} $ represents
the $k^{\text{th}}$ service time at queue $ i $, $ i=2,\ldots,n $, whereas
$ \tau_{1k} $ now denotes the interarrival time between the
$k^{\text{th}}$ customer and his predecessor in the external arrival stream.
At the initial time, all servers are assumed to be free of customers, and,
except for the first server, their buffers are empty; that is,
$ c_{i} = 0 $, $ i=2,\ldots,n $. Finally, we put $ c_{1} = \infty $ to
provide the model with the infinite arrival stream.

It is clear that we have to define the arrival epochs in the system as
\begin{equation}
a_{i}(k) = \begin{cases}
             \varepsilon,   & \text{if $ i=1$} \\
             d_{i-1}(k),    & \text{if $ i=2,\ldots,n$},
	   \end{cases} \label{tag7}
\end{equation}
for all $ k=1,2,\ldots $. Proceeding to vector notation, we get
$$
\bm{a}(k) = G \otimes \bm{d}(k), \qquad
\text{where} \qquad
G = \left(
      \begin{array}{ccccc}
        \varepsilon & \cdots & \cdots      & \varepsilon \\
	e	    & \ddots &             & \vdots      \\
		    & \ddots & \ddots      & \vdots \\
	\varepsilon &        & e           & \varepsilon
      \end{array}
    \right).
$$

With this vector representation, equation \eqref{tag3} takes the form
$$
\bm{d}(k) = \mathcal{T}_{k} \otimes G \otimes \bm{d}(k)
                                \oplus \mathcal{T}_{k} \otimes \bm{d}(k-1).
$$

The above equation can be considered as an implicit equation in
$ \bm{d}(k) $, having the form of \eqref{tag1} with
$ A = \mathcal{T}_{k} \otimes G $ and
$ \bm{b} = \mathcal{T}_{k} \otimes \bm{d}(k-1) $.
Moreover, one can readily see that the matrix
$$
\mathcal{T}_{k} \otimes G
= \left(
    \begin{array}{ccccc}
      \varepsilon & \cdots & \cdots      & \varepsilon \\
      \tau_{2k}   & \ddots &             & \vdots      \\
                  & \ddots & \ddots      & \vdots \\
      \varepsilon &        & \tau_{nk}   & \varepsilon
    \end{array}
  \right)
$$
looks just like that presented in Subsection~2.2 as an example; therefore, it
satisfies the condition of Lemma~1.

By applying Lemma~1, we obtain the solution
$$
\bm{d}(k)
= \sumo_{i=0}^{n-1} (\mathcal{T}_{k} \otimes G)^{i}
\otimes \mathcal{T}_{k} \otimes \bm{d}(k-1).
$$
Clearly, we have arrived at equation \eqref{tag5}, where the transition matrix
is defined as
$$
T_{k} = \sumo_{i=0}^{n-1} (\mathcal{T}_{k} \otimes G)^{i} \otimes \mathcal{T}_{k}.
$$
It is not difficult to verify that \cite{11}
\begin{multline}
T_{k}
\\
= \left( \begin{array}{rrrcc}
           \tau_{1k} & \varepsilon & \varepsilon & \ldots & \varepsilon \\
           \tau_{2k} \otimes \tau_{1k} & \tau_{2k} & \varepsilon & \ldots
                                                              & \varepsilon \\
           \vdots & \vdots &     & \ddots & \vdots \\
           \tau_{n-1k} \otimes \cdots \otimes \tau_{1k} &
           \tau_{n-1k} \otimes \cdots \otimes \tau_{2k} &
           \tau_{n-1k} \otimes \cdots \otimes \tau_{3k} &     & \varepsilon \\
           \tau_{nk} \otimes \cdots \otimes \tau_{1k} &
           \tau_{nk} \otimes \cdots \otimes \tau_{2k} &
           \tau_{nk} \otimes \cdots \otimes \tau_{3k} &  \ldots & \tau_{nk}
	 \end{array}
  \right). \label{tag8}
\end{multline}

\section{Tandem Queues with Finite Buffers and Blocking}

Suppose now that the buffers of the servers in the open tandem system
described above have limited capacity. Consequently, servers may be blocked
according to one of the blocking rules. In this paper, we restrict our
consideration to {\it manufacturing} blocking and {\it communication}
blocking, which are most commonly encountered in practice \cite{1,2}.

Consider an open system with $ n $ servers in tandem, and assume the
buffer at the $i^{\text{th}}$ server, $ i=2, \ldots, n$, to have capacity
$ b_{i}$, $ 0 \leq b_{i} < \infty $. We suppose that the buffer of the
first server, which is the input buffer of the system, is infinite. Below is
shown how the dynamics of tandem systems which operate according to both
manufacturing and communication blocking rules, can be described by state
equation \eqref{tag5}.

\subsection{Manufacturing Blocking}

Let us first suppose that the dynamics of the system follows the manufacturing
blocking rule. Under this type of blocking, if upon completion of a service,
the $i^{\text{th}}$ server sees the buffer of the $(i+1)^{\text{st}}$ server
full, the former server cannot be freed and has to remain busy until the
$(i+1)^{\text{st}}$ server completes its current service to provide a free
space in its buffer. Clearly, since the customers leave the system upon their
service completion at the $n^{\text{th}}$ server, this server cannot be
blocked.

It is not difficult to understand that the dynamics can be described by the
ordinary scalar equations \cite{1,2,6}
\begin{align*}
d_{i}(k) & = \max(\max(a_{i}(k),d_{i}(k-1))
                  + \tau_{ik}, d_{i+1}(k-b_{i+1}-1)),
                  \\
                  & \qquad i=1,\ldots,n-1, \\
d_{n}(k) & = \max(a_{n}(k),d_{n}(k-1)) + \tau_{nk},
\end{align*}
where $ a_{i}(k) $,  $ i=1,\ldots,n $, are still defined by \eqref{tag7}.
With max-algebra, one can readily rewrite these equations as \cite{11}
\begin{align*}
d_{i}(k) & = \tau_{ik} \otimes a_{i}(k) \oplus \tau_{ik} \otimes d_{i}(k-1)
                         \oplus d_{i+1}(k-b_{i+1}-1),
                         \\
                         & \qquad i=1,\ldots,n-1, \\
d_{n}(k) & = \tau_{nk} \otimes a_{n}(k) \oplus \tau_{nk} \otimes d_{n}(k-1).
\end{align*}

Assuming $ b_{i} = 0 $, $ i=2,\ldots,n $, for simplicity, we get the above
set of equations in the vector form
\begin{multline*}
\bm{d}(k) = \mathcal{T}_{k} \otimes \bm{a}(k)
                                 \oplus \mathcal{T}_{k} \otimes \bm{d}(k-1)
                                 \oplus G^{T} \otimes \bm{d}(k-1) \\
= \mathcal{T}_{k} \otimes \bm{a}(k)
                 \oplus (\mathcal{T}_{k} \oplus G^{T}) \otimes \bm{d}(k-1),
\end{multline*}
where $ G^{T} $ denotes the transpose of the above introduced matrix
$ G $. With $ \bm{a}(k) = G \otimes \bm{d}(k) $, we have the
equation
$$
\bm{d}(k)
= \mathcal{T}_{k} \otimes G \otimes \bm{d}(k)
  \oplus (\mathcal{T}_{k} \oplus G^{T}) \otimes \bm{d}(k-1).
$$

As in the case of the open tandem system with infinite buffers, we may apply
Lemma~1 to solve this equation for $ \bm{d}(k) $, and obtain
$$
\bm{d}(k) = \sumo_{i=0}^{n-1} (\mathcal{T}_{k} \otimes G)^{i}
                \otimes (G^{T} \oplus \mathcal{T}_{k}) \otimes \bm{d}(k-1),
$$
which equals \eqref{tag5} with the transition matrix
$$
T_{k} = \sumo_{i=0}^{n-1} (\mathcal{T}_{k} \otimes G)^{i}
                                          \otimes (G^{T} \oplus \mathcal{T}_{k}).
$$

Calculation of $ T_{k} $ \cite{11} leads us to
\begin{multline}
T_{k}
\\
= \left(
    \begin{array}{rrrcc}
      \tau_{1k} & e & \varepsilon & \ldots & \varepsilon \\
      \tau_{2k} \otimes \tau_{1k} & \tau_{2k} & e &  & \varepsilon \\
      \vdots & \vdots &     & \ddots &    \\
      \tau_{n-1k} \otimes \cdots \otimes \tau_{1k} &
      \tau_{n-1k} \otimes \cdots \otimes \tau_{2k} &
      \tau_{n-1k} \otimes \cdots \otimes \tau_{3k} &  & e \\
      \tau_{nk} \otimes \cdots \otimes \tau_{1k} &
      \tau_{nk} \otimes \cdots \otimes \tau_{2k} &
      \tau_{nk} \otimes \cdots \otimes \tau_{3k} & \ldots & \tau_{nk}
    \end{array}
  \right). \label{tag9}
\end{multline}
Note that the matrices in \eqref{tag8} and \eqref{tag9} differ only in the
elements of the upper diagonal adjacent to the main diagonal, which become
equal to $ e $ in \eqref{tag9}.

Let us now derive equation \eqref{tag5} for the system with the capacity of each
buffer $ b_{i} = 1 $, $ i=2,\ldots,n $. Clearly, the dynamics of the
system can be described by the equation
$$
\bm{d}(k) = \mathcal{T}_{k} \otimes G \otimes \bm{d}(k)
                                 \oplus \mathcal{T}_{k} \otimes \bm{d}(k-1)
                                 \oplus G^{T} \otimes \bm{d}(k-2).
$$
The application of Lemma~1 leads us to the equation
$$
\bm{d}(k) = \sumo_{i=0}^{n-1} (\mathcal{T}_{k} \otimes G)^{i}
                               \otimes (\mathcal{T}_{k} \otimes \bm{d}(k-1)
                               \oplus G^{T} \otimes \bm{d}(k-2)).
$$
Finally, with the notations
$$
\widetilde{\bm{d}}(k)
= \left(
    \begin{array}{l}
      \bm{d}(k) \\
      \bm{d}(k-1)
    \end{array}
  \right) \qquad \text{and} \qquad
S_{k} = \sumo_{i=0}^{n-1} (\mathcal{T}_{k} \otimes G)^{i},
$$
we arrive at \eqref{tag5}, taking the form
$$
\widetilde{\bm{d}}(k)
= \left(
    \begin{array}{ccccc}
      S_{k} \otimes \mathcal{T}_{k} & S_{k} \otimes G^{T} \\
      E                                 & \mathcal{E}
    \end{array}
  \right) \otimes \widetilde{\bm{d}}(k-1).
$$

\subsection{Communication Blocking}

We now turn to a brief discussion of systems operating under the communication
blocking rule. This type of blocking requires a server not to initiate
service of a customer if the buffer of the next server is full. In this case,
the server remains unavailable until the current service at the next server is
completed.

Let us suppose that the above system follows communication blocking. The
dynamics of this system can be described by the equations \cite{1,6}
\begin{align*}
d_{i}(k) & = \max(a_{i}(k),d_{i}(k-1),d_{i+1}(k-b_{i+1}-1)) + \tau_{ik},
\\
& \qquad i=1,\ldots,n-1, \\
d_{n}(k) & = \max(a_{n}(k),d_{n}(k-1)) + \tau_{nk},
\end{align*}
or, equivalently, by the max-algebra equations
\begin{align*}
d_{i}(k) & = \tau_{ik} \otimes a_{i}(k) \oplus \tau_{ik} \otimes d_{i}(k-1)
       \oplus \tau_{ik} \otimes d_{i+1}(k-b_{i+1}-1),
       \\
       & \qquad i=1,\ldots,n-1, \\
d_{n}(k) & = \tau_{nk} \otimes a_{n}(k) \oplus \tau_{nk} \otimes d_{n}(k-1).
\end{align*}

For a particular system with $ b_{i} = 0 $, $ i=2,\ldots,n $, we can write
in the same manner as for manufacturing blocking
\begin{multline*}
\bm{d}(k) = \mathcal{T}_{k} \otimes \bm{a}(k)
                                 \oplus \mathcal{T}_{k} \otimes \bm{d}(k-1)
                \oplus \mathcal{T}_{k} \otimes G^{T} \otimes \bm{d}(k-1) \\
= \mathcal{T}_{k} \otimes G \otimes \bm{d}(k) \oplus
              \mathcal{T}_{k} \otimes (E \oplus G^{T}) \otimes \bm{d}(k-1).
\end{multline*}
Furthermore, we have the solution
$$
\bm{d}(k) = \sumo_{i=0}^{n-1} (\mathcal{T}_{k} \otimes G)^{i} \otimes
              \mathcal{T}_{k} \otimes (E \oplus G^{T}) \otimes \bm{d}(k-1),
$$
which takes the form of \eqref{tag5} with the matrix
\begin{multline}
T_{k} = \sumo_{i=0}^{n-1} (\mathcal{T}_{k} \otimes G)^{i} \otimes
                                      \mathcal{T}_{k} \otimes (E \oplus G^{T}) \\
= \left(
    \begin{array}{rrrcl}
     \tau_{1k} & \tau_{1k} & \varepsilon & \ldots & \varepsilon \\
     \tau_{2k} \otimes \tau_{1k} & \tau_{2k} \otimes \tau_{1k} &
                                                 \tau_{2k} &  & \varepsilon \\
     \vdots & \vdots & \vdots & \ddots &  \\
     \tau_{n-1k} \otimes \cdots \otimes \tau_{1k} &
     \tau_{n-1k} \otimes \cdots \otimes \tau_{1k} &
     \tau_{n-1k} \otimes \cdots \otimes \tau_{2k} & \ldots &
                                                                \tau_{n-1k} \\
     \tau_{nk} \otimes \cdots \otimes \tau_{1k} &
     \tau_{nk} \otimes \cdots \otimes \tau_{1k} &
     \tau_{nk} \otimes \cdots \otimes \tau_{2k} & \ldots &
                                                 \tau_{nk} \otimes \tau_{n-1k}
    \end{array}
  \right).
\label{tag10}
\end{multline}

Let us finally consider the system with the capacity of each buffer
$ b_{i}=1 $, $ i=2,\ldots,n $. As for manufacturing blocking, we may
represent the dynamics of the system by the equation
$$
\widetilde{\bm{d}}(k)
= \left(
    \begin{array}{ccccc}
      S_{k} \otimes \mathcal{T}_{k}
                          & S_{k} \otimes \mathcal{T}_{k} \otimes G^{T} \\
      E                   & \mathcal{E}
    \end{array}
  \right) \otimes \widetilde{\bm{d}}(k-1).
$$

\section{Representation of Tandem System Performance}

In this section, we show how performance measures of open tandem systems may
be represented based on the max-algebra models. We consider the measures
representing the system time and the waiting time of customers, and give
the corresponding linear max-algebra equations which allow us to describe and
calculate these criteria in a simple way. The ordinary representation of other
measures which one normally chooses in the analysis of queueing systems,
including the utilization of a server, the number of customers at a queue, and
the queue length, involves the operation of division \cite{1,3,4,6}, and they
therefore cannot be expressed through linear equations in max-algebra.

\subsection{The System Time of Customers}

For the open tandem system with infinite buffers described above, let us
define the vector of the system (sojourn) times of the $k^{\text{th}}$
customer as $ \bm{s}(k) = (s_{1}(k),\ldots,s_{n}(k))^{T} $, where
$ s_{i}(k) $ denotes the time required for the customer to pass through
all queues up to and including queue $ i $. Since the first queue represents
the external arrival stream of customers, we do not have to include the time
spent in this queue, and we therefore have $ s_{1}(k) = 0 $. Furthermore,
the above definition of the system time leads to the equations
$$
s_{i}(k) = d_{i}(k) - d_{1}(k), \quad i=1,\ldots,n.
$$
With max-algebra vector notations, we may rewrite these equations as
\begin{equation}
\bm{s}(k) = \bm{d}(k) \otimes d_{1}^{-1}(k). \label{tag11}
\end{equation}
It follows from \eqref{tag5} and the equality
$ d_{1}(k) = d_{1}(k-1) \otimes \tau_{1k} $ that
$$
\bm{s}(k) = T_{k} \otimes \bm{d}(k-1)
                               \otimes d_{1}^{-1}(k-1) \otimes \tau_{1k}^{-1}.
$$
By applying \eqref{tag11} with $ k $ replaced by $ k-1 $, we finally have
\begin{equation}
\bm{s}(k) = U_{k} \otimes \bm{s}(k-1), \label{tag12}
\end{equation}
where $ U_{k} = \tau_{1k}^{-1} \otimes T_{k} $.

As it is easy to see, the above dynamic equation is appropriate for
representing the system time in open tandem systems with both infinite and
finite buffers. For particular systems, equation \eqref{tag12} will differ only
in the matrix $ T_{k} $, and thus in the matrix $ U_{k} $, inherent in
their associated dynamic state equations. Finally, note that a lower
triangular transition matrix $ T_{k} $ will result in a matrix
$ U_{k} $ of the same kind. One can consider the open tandem system with
infinite buffers as an example.

\subsection{The Waiting Time of Customers}

Consider the open tandem system with infinite buffers once again. The system
time of a customer in this system consists of his service time and the waiting
time. Therefore, introducing the symbol $ w_{i}(k) $ for the
$k^{\text{th}}$ customer to denote the total time spent on waiting for service
at the queues from 1 to $ i $, we have the equations
\begin{align*}
s_{1}(k) & = w_{1}(k) = 0, \\
s_{i}(k) & = w_{i}(k) + \sum_{j=2}^{i} \tau_{jk}, \quad i=2,3,\ldots,n.
\end{align*}

In order to represent the relation between the system and waiting times in
max-algebra vector form, let us further introduce the vector
$ \bm{w}(k) = (w_{1}(k),\ldots,w_{n}(k))^{T} $, and the diagonal
matrix
$$
P_{k}
= \left(
    \begin{array}{ccccc}
      \tau_{1k}   & \varepsilon                 & \ldots & \varepsilon \\
      \varepsilon & \tau_{1k} \otimes \tau_{2k} &        & \varepsilon \\
      \vdots      &                             & \ddots & \\
      \varepsilon & \varepsilon                 &        &
                                    \tau_{1k} \otimes \cdots \otimes \tau_{nk}
    \end{array}
  \right).
$$
Clearly, we may now write
$$
\bm{s}(k) = \tau_{1k}^{-1} \otimes P_{k} \otimes \bm{w}(k).
$$
Since the multiplicative inverse exists for the matrix $ P_{k} $, it is not
difficult to resolve the above equation for $ \bm{w}(k) $:
$$
\bm{w}(k) = \tau_{1k} \otimes P_{k}^{-1} \otimes \bm{s}(k).
$$
With \eqref{tag12}, we successively obtain
\begin{multline*}
\bm{w}(k)
= \tau_{1k} \otimes P_{k}^{-1} \otimes \tau_{1k}^{-1} \otimes T_{k}
                                                  \otimes \bm{s}(k-1) \\
= P_{k}^{-1} \otimes T_{k} \otimes \tau_{1k-1}^{-1} \otimes P_{k-1}
                               \otimes \left( \tau_{1k-1} \otimes P_{k-1}^{-1}
                                  \otimes \bm{s}(k-1) \right) \\
= \tau_{1k-1}^{-1} \otimes P_{k}^{-1} \otimes T_{k} \otimes P_{k-1}
                                                    \otimes \bm{w}(k-1).
\end{multline*}

Finally, we arrive at the dynamic equation
\begin{equation}
\bm{w}(k) = V_{k,k-1} \otimes \bm{w}(k-1), \label{tag13}
\end{equation}
with the transition matrix
$ V_{k,k-1}
        = \tau_{1k-1}^{-1} \otimes P_{k}^{-1} \otimes T_{k} \otimes P_{k-1} $.

Note that, in general, equation \eqref{tag13} can be extended to tandem queues
with finite buffers and blocking. In that case, however, the quantities
$ w_{i}(k) $ will include not only the time spent on waiting for service,
but also the blocking time of servers.

\section{Serial and Parallel Simulation of Tandem Queues}

In this section, we briefly discuss serial and parallel simulation algorithms
based on the algebraic models presented above, and outline their performance.
We take, as the starting point, the state equation
\begin{equation}
\bm{d}(k) = T_{k} \otimes \bm{d}(k-1). \label{tag14}
\end{equation}
with the matrix $ T_{k} $ defined by \eqref{tag8}. Clearly, calculations
with other matrices, including performance evaluation via equations
\eqref{tag12} and \eqref{tag13}, will normally differ little in complexity.

Furthermore, we assume as in \cite{1,2,7} that the service times
$ \tau_{ik} $, which are normally defined in queueing system simulation as
realizations of given random variables, are available for all
$ i=1,\ldots,n $ and $ k=1,2\ldots $, when required. We will therefore
concentrate only on procedures of evaluating the system state vectors, based
on \eqref{tag14}.

\subsection{Simulation with a Scalar Processor}

Let us first consider time and memory requirements when only one scalar
processor is available for simulation of tandem systems. The related serial
algorithm consists of consecutive steps, evaluating new state vectors. The
$k^{\text{th}}$ step involves determination of the transition matrix
$ T_{k} $ and multiplication of this matrix by the vector
$ \bm{d}(k-1) $ to produce the vector $ \bm{d}(k) $.
However, particular algorithms designed for computation with matrices
\eqref{tag6}, \eqref{tag8}, \eqref{tag9}, and \eqref{tag10} may execute the step in
different ways, according to the structure of the matrices, so as to reduce
time and memory costs.

Assume that a system is simulated until the $K^{\text{th}}$ service completion
at server $ n $, and denote the overall number of the operations
$ \oplus $ and $ \otimes $ to be performed within the $k^{\text{th}}$
step to evaluate the matrix $ T_{k} $ and to compute the product
$ T_{k} \otimes \bm{d}(k-1) $ respectively by $ N_{1} $ and
$ N_{2} $. In that case, the entire simulation algorithm will require
$ N = K(N_{1}+N_{2}) $ operations, ignoring index manipulations. Finally,
we denote by $ M $ the number of memory locations involved in the
computations.

Let us now consider the open tandem system with infinite buffers, and denote
the entries of its transition matrix $ T_{k} $ by $ t_{ij}^{(k)} $,
$ i=1,\ldots,n $, $ j=1,\ldots,n $. Taking into account the lower
triangular form of the matrix $ T_{k} $ defined by \eqref{tag8}, a serial
algorithm for calculating $ K $ successive vectors $ \bm{d}(k) $
may be readily designed as follows.

\begin{algorithm}{1}
For $ i=1, \ldots, n, $ do $ d_{i}(0) \longleftarrow \varepsilon .$
                                                                      \\
For  $ k=1, \ldots, K, $ do \\
\hphantom{For} for $ i=1, \ldots, n, $ do \\
\hphantom{For for} $ t_{ii}^{(k)} \longleftarrow \tau_{ik}; $ \\
\hphantom{For for} for $ j=1, \ldots, i-1, $ do \\
\hphantom{For for for} $ t_{ij}^{(k)} \longleftarrow t_{i-1j}^{(k)}
                                                 \otimes \tau_{ik}; $ \\
\hphantom{For} for $ i=1, \ldots, n, $ do \\
\hphantom{For for} $ d_{i}(k)
                    \longleftarrow t_{i1}^{(k)} \otimes d_{i}(k-1); $ \\
\hphantom{For for} for $ j=2, \ldots, i, $ do \\
\hphantom{For for for} $ d_{i}(k) \longleftarrow d_{i}(k)
                                     \oplus t_{ij}^{(k)} \otimes d_{j}(k-1). $
\end{algorithm}

It is not difficult to see that for the algorithm, $ N_{1}=n(n+1)/2 $
and $ N_{2} = n^{2} $ operations are required; it entails
$ M = n(n+5)/2 $ memory locations. With the total number of operations
$ N = O(n(3n+1)K/2) $ involved in calculating $ K $ state vectors, the
algorithm proves to have a performance comparable to other serial simulation
procedures \cite{1,7}. As an example, the serial algorithm in \cite{7},
designed for simulation of open tandem systems with infinite buffers can be
considered. This algorithm allows one to compute the $K^{\text{th}}$ departure
time at queue $ n $ in $ O(2K(n+1)) $ operations; however, it is
actually a scalar algorithm intended to obtain only the value of
$ d_{n}(k) $ rather than the whole vector $ \bm{d}(k) $.
Therefore, one has to multiply its performance criterion by $ n $ so as to
provide a proper basis for a comparison between algorithms. It should be noted
in conclusion that calculations with matrix \eqref{tag6} can be performed using
only $ N = O(2Kn) $ operations and $ M = O(3n) $ memory locations.

\subsection{Simulation with a Vector Processor}

Suppose now that simulation is executed on a vector processor equipped with
vector registers of a length large enough for the processing of $n$-vectors.
With the notations
$ \bm{t}_{i}^{(k)} = (t_{i1}^{(k)},\ldots,t_{in}^{(k)}) $,
$ i=1,\ldots,n $, we may write the following vector modification of
Algorithm~1:

\begin{algorithm}{2}
$ \bm{d}(0)
             \longleftarrow (\varepsilon, \ldots, \varepsilon)^{T}. $ \\
For $ k=1,\ldots,K, $ do \\
\hphantom{For} $ T_{k} \longleftarrow \mathcal{E}; $ \\
\hphantom{For} for $ i=1,\ldots,n, $ do \\
\hphantom{For for} $ t_{ii}^{(k)} \longleftarrow \tau_{ik}; $ \\
\hphantom{For for} $ \bm{t}_{i}^{(k)} \longleftarrow
                        \bm{t}_{i-1}^{(k)} \otimes \tau_{ik}; $ \\
\hphantom{For for} $ d_{i}(k) \longleftarrow \bm{t}_{i}^{(k)}
                                                  \otimes \bm{d}(k-1). $
\end{algorithm}

First note that implementation of the vector processor makes it possible to
compute matrix \eqref{tag8} by performing only $ N_{1} = n $ vector
operations. Furthermore, evaluation of each element of the vector
$ \bm{d}(k) $ from \eqref{tag14} actually involves both componentwise
addition of a row of the matrix $ T_{k} $ to the vector
$ \bm{d}(k-1) $, and determination of the maximum over the elements of
this vector sum.

With the vector processor, one can add two vectors together in a single
operation. It follows from the triangular form of \eqref{tag8} that to compute
the $i^{\text{th}}$ element of $ \bm{d}(k) $, one actually has to
perform no more than $ i $ maximizations. By applying the {\it recursive
doubling method\/} \cite{13}, we may obtain the maximum over $ i $
consecutive elements of a vector in $ \log_{2}i  $ operations. For all
entries of the vector $ \bm{d}(k) $, we therefore get
$ N_{2} = n+\log_{2}1 + \cdots + \log_{2}n = n+\log_{2}(n!) $.

Finally, evaluation of $ K $ state vectors requires
$ N = O(K(\log_{2}(n!)+2n)) $ operations of a vector processor. As is
easy to see, vector computations allow us to achieve the speedup
$ S_{v} = O(n(3n+1)/(\log_{2}(n!)/2+n)) $ in relation to the sequential
procedure discussed in the previous subsection.

\subsection{Parallel Simulation of Tandem Systems}

We conclude this section with the discussion of a parallel simulation
algorithm intended for implementation on single instruction, multiple data
(SIMD) parallel processors. Other parallel algorithms based on the algebraic
models can be found in \cite{2}.

Let $ P $ be the number of processors available, and $ P \geq n $. The
algorithm consists of $ L = \lceil K/P \rceil $ steps, where
$ \lceil x \rceil $ denotes the smallest integer greater than or equal
to $ x $, and it can be described as follows.

\begin{algorithm}{3}
$ \bm{d}(0)
             \longleftarrow (\varepsilon, \ldots, \varepsilon)^{T}. $ \\
for $ l=1, \ldots, L, $ do \\
\hphantom{For} in parallel,
               for $ i=(l-1)P+1, (l-1)P+2, \ldots, \min(lP,K), $ do \\
\hphantom{For in parallel, for} {\it evaluate\/} $ T_{i}; $ \\
\hphantom{For} for $ i=(l-1)P+1, (l-1)P+2, \ldots, \min(lP,K), $ do \\
\hphantom{For for} in parallel, for $ j=1, \ldots, n, $ do
$ d_{j}(i) \longleftarrow \bm{t}_{j}^{(i)}
                                                  \otimes \bm{d}(i-1). $
\end{algorithm}

To estimate the performance of this algorithm, first note that each step
starts with parallel evaluation of $ P $ consecutive transition matrices,
which entails $ N_{1} = n(n+1)/2 $ parallel operations. Next follows the
determination of $ P $ consecutive state vectors, with each vector
evaluated in parallel by computing its elements on separate processors. Since
evaluation of a vector element involves $ n $ ordinary additions and the
same number of maximizations, one has to perform $ 2n $ parallel
operations so as to determine the whole vector. Therefore, we have
$ N_{2} = 2Pn $ operations required for computing $ P $ vectors.

Finally, the entire algorithm entails $ N = L(n(n+1)/2+2Pn) $ parallel
operations. It is not difficult to understand that with $ P \geq n $, we
get $ N=O(3Kn/2) $ as $ K \to \infty $. Moreover, by comparison with
the performance of the above sequential procedure, one can conclude that for
$ P = n $ and $ K $ sufficiently large, the parallel algorithm
achieves the speedup $ S_{P} = O(3P/5) $.

\section{Conclusions}

The max-algebra approach provides a very convenient way of representing
tandem queueing system models. The algebraic models describe the dynamics of
the systems through linear max-algebra equations similar to those being
studied in the conventional linear systems theory. The approach therefore
offers the potential for extending classical results of both conventional
linear algebra and systems theory to analyze queueing systems. Moreover,
it provides the basis for applying methods and algorithms of numerical
algebra to the development of efficient procedures for the queueing systems
simulation, including the algorithms designed for implementation on parallel
and vector processors.

Since the closed form representation of the dynamics of more complicated
queueing systems including the $ G/G/m $ queue and queueing networks,
normally involves three operations, namely addition, maximization, and
minimization (see, e.g., \cite{5,6}), these systems cannot be described in
terms of max-algebra. In that case, however, appropriate algebraic
representations can be obtained using {\it minimax algebra\/} \cite{8,10}.

\bibliographystyle{utphys}

\bibliography{A_max-algebra_approach_to_modeling_and_simulation_of_tandem_queueing_systems}

\providecommand{\href}[2]{#2}\begingroup\raggedright\begin{thebibliography}{10}

\bibitem{1}
L.~Chen and C.-L. Chen, \href{http://dx.doi.org/10.1109/WSC.1990.129572}{``A
  fast simulation approach for tandem queueing systems,''} in {\em 1990 Winter
  Simulation Conference Proceedings}, O.~Balci, R.~P. Sadowski, and R.~E.
  Nance, eds., pp.~539--546.
\newblock IEEE, 1990.

\bibitem{2}
A.~G. Greenberg, B.~D. Lubachevsky, and I.~Mitrani, ``Algorithms for
  unboundedly parallel simulations,''
  \href{http://dx.doi.org/10.1145/128738.128739}{{\em ACM Trans. Comput. Syst.}
  {\bfseries 9} no.~3, (1991) 201--221}.

\bibitem{3}
N.~Krivulin, ``Unbiased estimates for gradients of stochastic network
  performance measures,'' \href{http://dx.doi.org/10.1007/BF00995493}{{\em Acta
  Appl. Math.} {\bfseries 33} (1993) 21--43},
  \href{http://arxiv.org/abs/1210.5418}{{\ttfamily arXiv:1210.5418 [math.OC]}}.

\bibitem{4}
N.~K. Krivulin, ``Optimization of complex systems for simulation modeling,''
  {\em Vestnik Leningrad Univ. Math.} {\bfseries 23} no.~2, (1990) 64--67.

\bibitem{5}
N.~K. Krivulin, ``A recursive equations based representation for the g/g/m
  queue,'' \href{http://dx.doi.org/10.1016/0893-9659(94)90116-3}{{\em Appl.
  Math. Lett.} {\bfseries 7} no.~3, (1994) 73--77},
  \href{http://arxiv.org/abs/1210.6012}{{\ttfamily arXiv:1210.6012 [math.OC]}}.

\bibitem{6}
N.~K. Krivulin, ``Recursive equations based models of queueing systems,'' in
  {\em Proc. Europ. Simulation Symp. ESS 94}, pp.~252--256.
\newblock SCSI, San Diego, CA, 1994.
\newblock \href{http://arxiv.org/abs/1210.7445}{{\ttfamily arXiv:1210.7445
  [math.NA]}}.

\bibitem{7}
S.~M. Ermakov and N.~K. Krivulin, ``Efficient algorithms for tandem queueing
  system simulation,''
  \href{http://dx.doi.org/10.1016/0893-9659(94)90092-2}{{\em Appl. Math. Lett.}
  {\bfseries 7} no.~6, (1994) 45--49},
  \href{http://arxiv.org/abs/1210.6378}{{\ttfamily arXiv:1210.6378 [math.NA]}}.

\bibitem{8}
R.~Cuninghame-Green, {\em Minimax Algebra}, vol.~166 of {\em Lecture Notes in
  Economics and Mathematical Systems}.
\newblock Springer, Berlin, 1979.

\bibitem{9}
G.~Cohen, P.~Moller, J.-P. Quadrat, and M.~Viot, ``Algebraic tools for the
  performance evaluation of discrete event systems,''
  \href{http://dx.doi.org/10.1109/5.21069}{{\em Proc. IEEE} {\bfseries 77}
  no.~1, (January, 1989) 39--85}.

\bibitem{10}
R.~A. Cuninghame-Green, ``Minimax algebra and applications,''
  \href{http://dx.doi.org/10.1016/0165-0114(91)90130-I}{{\em Fuzzy Sets and
  Syst.} {\bfseries 41} no.~3, (1991) 251--267}.

\bibitem{11}
N.~K. Krivulin, ``Using max-algebra linear models in the representation of
  queueing systems,'' in {\em Proc. 5th SIAM Conf. on Applied Linear Algebra},
  J.~G. Lewis, ed., pp.~155--160.
\newblock SIAM, Philadelphia, 1994.
\newblock \href{http://arxiv.org/abs/1210.6019}{{\ttfamily arXiv:1210.6019
  [math.OC]}}.

\bibitem{12}
G.~J. Olsder, ``About difference equations, algebras and discrete events,''
  vol.~31 of {\em CWI syllabus}, pp.~180--209.
\newblock CWI, Amsterdam, 1992.

\bibitem{13}
J.~M. Ortega, {\em Introduction to Parallel and Vector Solution of Linear
  Systems}.
\newblock Frontiers in Computer Science. Plenum Press, New York, 1988.

\end{thebibliography}\endgroup

\end{document}